\documentclass[12pt,oneside]{amsart}
\usepackage[english]{babel}
\usepackage{amssymb}
\usepackage{amsthm}
\usepackage{mathtools}

\usepackage[foot]{amsaddr}
\usepackage{amsmath}
\usepackage{a4wide}
\usepackage{tikz}
\usetikzlibrary{calc}
\usepackage{nccmath}
\usepackage{mathrsfs}
\usepackage{esvect}
\usepackage{ulem}
\usepackage{enumitem}
\usepackage{hyperref}
\usepackage{cleveref}
\usepackage{subfiles}
\hypersetup{
    colorlinks,
    linkcolor=red,
    citecolor=black,
    urlcolor=blue
}
\usepackage{verbatim}
\usepackage[margin=0.5in]{geometry}
\usepackage[parfill]{parskip}   
\usepackage[numbers,sort&compress]{natbib}
\usepackage{tikz}
\usepackage{caption}
\usepackage{subcaption}
\usepackage[section]{placeins}
\usepackage{calrsfs}
\usetikzlibrary{arrows,positioning}
\usetikzlibrary{graphs}
\usetikzlibrary{graphs.standard}

\makeatletter
\def\subsubsection{\@startsection{subsubsection}{3}%
  \z@{.5\linespacing\@plus.7\linespacing}{.1\linespacing}%
  {\normalfont\itshape}}
\makeatother

\newtheorem{thm}{Theorem}[section]

\newtheorem{conjecture}[thm]{Conjecture}

\newtheorem{proposition}[thm]{Proposition}

\newtheorem{problem}[thm]{Problem}

\newtheorem{corollary}[thm]{Corollary}
\newtheorem{question}[thm]{Question}

\newtheorem{claim}[thm]{Claim}
\newtheorem{proof4.2}[thm]{Proof of Corollary 4.2}

\newcommand{\ignore}[1]{}

\begin{document}

\author{Rebekah Herrman}
\address[Rebekah Herrman]{Department of Industrial and Systems Engineering, The University of Tennessee, Knoxville, TN}
\email[Rebekah Herrman]{rherrma2@tennessee.edu}

\title[Upper bound for the $(d-2)$-leaky forcing number of $Q_d$ and $\ell$-leaky forcing number of $GP(n,1)$] 
{Upper bound for the $(d-2)$-leaky forcing number of $Q_d$ and $\ell$-leaky forcing number of $GP(n,1)$}

\linespread{1.3}
\pagestyle{plain}

\begin{abstract}
Leaky-forcing is a recently introduced variant of zero-forcing that has been studied for families of graphs including paths, cycles, wheels, grids, and trees. In this paper, we extend previous results on the leaky forcing number of the d-dimensional hypercube, $Q_d$, to show that the $(d-2)$-leaky forcing number of $Q_d$ is at most $2^{d-1}$. We also examine a question about the relationship between the size of a minimum $\ell$-leaky-forcing set and a minimum zero-forcing set for a graph $G$.
\end{abstract}
\maketitle 

\section{Introduction}

The zero-forcing game on a simple graph $G = (V,E)$ is a vertex coloring game in which an initial set of vertices is colored blue, and colored vertices ``force'' uncolored ones, which is to say colors them, according to the \textit{color change rule}. The color change rule states that a colored vertex $v$ can color an uncolored vertex $u$ if $u$ is the only uncolored neighbor of $v$. If an initial set of vertices, $B \subset V(G)$, is colored and every vertex in $G$ can be eventually colored after successively applying the color change rule, $B$ is called a zero-forcing set. The objective of the zero-forcing game is to find the smallest zero-forcing set. 

Zero-forcing has several applications in mathematics and the physical sciences. For example, in linear algebra, it can be used to bound the maximum nullity and minimum rank problem \cite{work2008zero, barioli2010zero, huang2010minimum}. It has also been shown to be related to different types of domination, such as power domination and Grundy domination, as well as the general position number of a graph \cite{bozeman2019restricted, lin2019zero, benson2018zero, hua2021zero}. In computer science, it has applications to fast-mixed searching, while in physics, it has applications to quantum control \cite{fallat2016complexity, burgarth2007full}.

Since zero-forcing can be used to model flow through networks, a motivating idea is to determine how flow changes when there are leaks in the network. Thus, the leaky-forcing variant of zero-forcing was introduced in \cite{dillman2019leaky}. In this version, once the zero-forcing set has been colored, leaks are added to vertices of $G$, and a vertex that has a leak added to it is said to be \textit{leaky}. A vertex with a leak does not enforce the color change rule, so a leak can be thought of as a vertex connected to the leaky vertex by a single edge. The $\ell$-forcing number, $Z_{(\ell)}(G)$, is the size of the minimum $\ell$-leaky forcing set, which is the set that can force $G$ even with the addition of $\ell$ leaks. When the context is clear, we shall call the set of vertices of the minimum $\ell$-leaky forcing set $B_{\ell}$. Note the case when $\ell=0$ is the zero-forcing game. The $\ell$-leaky forcing number of paths, cycles, wheels, grids, and trees has been determined, and variants called edge-leaky-forcing, specified-leaky forcing, and mixed-leaky forcing have been introduced, as well \cite{dillman2019leaky, alameda2020generalizations, alameda2020leaky}.

In this paper, we shall concentrate on the leaky-forcing number of the $d$-dimensional discrete hypercube, $Q_d$. The vertex set of $Q_d$ is the set of $0-1$ valued sequences of length $d$, $\epsilon =(\epsilon_i)_{i=1}^d$, where $\epsilon_i = 0$ or $1$. See Fig.~\ref{fig:cube} for $Q_3$. Two sequences are joined by an edge if they differ in exactly one place. Recently, Dillman and Kenter examined the bounds for $Z_{(\ell)}(Q_d)$ and found that  $Z_{(1)}(Q_3) = 4$, $Z_{(2)}(Q_4) = 8$, and $Z_{(3)}(Q_5) = 16$, however they do not extend the result to higher dimensions \cite{dillman2019leaky}. Our main result in this paper is such an extension.

\begin{thm}\label{thm:cube}
$Z_{(d-2)}(Q_d) \leq 2^{d-1}$ for $d \geq 2$.
\end{thm}

This paper is organized as follows. In Sec.~\ref{sec:cube}, we prove Theorem~\ref{thm:cube}. After this, in Section~\ref{sec:prism}, we turn to the prism graph $GP(n,1)$. The prism graph is obtained by taking two vertex-disjoint cycles on $n$ vertices, $u_1, ..., u_n$ and $x_1, ... , x_n$ and adding the $n$ edges $u_1x_1, ..., u_nx_n$. We refer to the former cycle as a whole as $U$ and the latter cycle as $X$. The graph $GP(3,1)$ can be seen in Fig.~\ref{fig:prism}. In this section, we prove the following result.

\begin{thm}\label{thm:prism}

\begin{align*}
Z_{(\ell)}(GP(n,1)) = & \begin{cases}
 \; 3\text{ if } \textup{ $ \ell \in \{0,1\}$ and $n =3$ } \\
  \; 4\text{ if } \textup{ $ \ell \in \{0,1\}$ and $n \geq 4$ or $\ell=2$ and $n=3$} \\
  \;2n \text{ if } \textup{ $ \ell \geq 3$}
\end{cases}\\.
\end{align*}

Furthermore, $Z_{(2)}(GP(3,1)) = 4$, $Z_{(2)}(GP(4,1)) \leq 6$ and $Z_{(2)}(GP(n,1)) \leq n$ for $n > 4$.
\end{thm}

We then examine the following question from \cite{alameda2020leaky} and provide partial results in Sec.~\ref{sec:question}.

\begin{question}\label{mainquestion}\cite{alameda2020leaky}
Given a graph, is there a minimum $\ell$-leaky forcing set $B_{\ell}$ that contains a minimum zero-forcing set $B$?
\end{question}
\noindent The conclusion follows in Sec.~\ref{sec:conclusion}.

\section{Proof of \ref{thm:cube}}\label{sec:cube}

%The $d$-dimensional hypercube is a graph that consists of $2^d$ vertices. Each vertex is assigned a unique binary label from $00...0$ to $11...1$, and two vertices are adjacent if and only if their binary label differs in exactly one digit. Thus, each vertex has degree $d$. An example of the $3$-dimensional hypercube is shown in Figure~\ref{fig:cube}. With this information, we will now prove Theorem~\ref{thm:cube}.
\begin{figure}
\begin{tikzpicture}
	\begin{scope}[every node/.style={scale=0.5,circle,fill=black,draw}]
    \node (A) at (0,0) {};
    \node (B) at (0,2) {};
	\node (C) at (2,0) {}; 
	\node (D) at (2,2) {};
	\node (E) at (1,2.6) {}; 
	\node (F) at (1,0.6) {};
	\node (G) at (3,.6) {}; 
	\node (H) at (3,2.6) {};

\end{scope}

\begin{scope}[every node/.style={scale=0.6}]
\node at (-.25,-0.25) {000};
\node at (-.25,2.25) {001};
\node at (2.25,-.25) {010};
\node at (0.65,0.7) {100};
\node at (0.65,2.7) {101};
\node at (3.4,2.7) {111};
\node at (3.4,0.7) {110};
\node at (1.8,2.25) {011};

\end{scope}
		
%\draw    (A) to[out=-20,in=20] (C);
%	\draw    (B) to[out=-20,in=20] (D);

	\draw (A) -- (B);
	\draw (A) -- (C);
	\draw (A) -- (F);
	\draw (B) -- (D);
	\draw (B) -- (E);
	\draw (C) -- (D);
	\draw (C) -- (G);
	\draw (D) -- (H);
	\draw (E) -- (F);
	\draw (E) -- (H);
	\draw (F) -- (G);
	\draw (G) -- (H);

\end{tikzpicture}
\caption{The 3-dimensional hypercube, $Q_3$.}\label{fig:cube}
\end{figure}
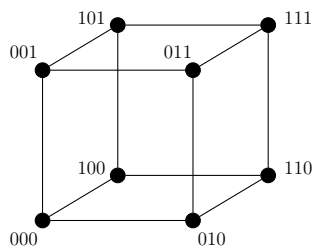

The case $d=2$ follows from the zero-forcing case for cycles, and in any case is trivial. Consider $d \geq 3$. We shall denote the set of all sequences with a 0 in the first position $Q$ and the set of sequences with 1 in the first position $Q'$. Each of these sets has $2^{d-1}$ vertices and spans a $(d-1)$-hypercube. We will show that $Q$ is a $(d-2)$-leaky forcing set. %Once $Q$ has been shown to be a $(d-2)$-leaky forcing set, by symmetry, $Q'$ will also be one. 

Color the vertices of $Q$ blue and apply the color change rule once. After the first application, at most $d-2$ vertices will remain uncolored, and all of these vertices are in $Q'$. If there are no uncolored vertices, the game ends, so suppose there are $u > 0$ uncolored vertices after the first application of the color change rule. Call these uncolored vertices $v_1, ..., v_u$. If we show an arbitrary vertex in this set can be forced, the proof is complete. In order to show that an arbitrary vertex, $v_i$, in this set can be forced, we will show that $v_i$ has at least one neighbor that is not leaky and $v_i$ is the only uncolored neighbor of the non-leaky vertex.

Pick an uncolored vertex, say $v_1$. In $Q'$, every vertex has degree $d-1$, and there can be at most $d-2$ uncolored vertices in $Q'$. Hence, $v_1$ has $c \geq (d-1)-(d-2)+1 = 2$ colored neighbors in $Q'$, where the ``$+1$'' comes from the fact that $v_1$ is not colored. Let $x_1, ... , x_c$ be the colored neighbors of $v_1$. Our goal is to show that one of these vertices can force $v_1$.

Note that $v_1$ has $d-1-c$ uncolored neighbors in $Q'$ since it has degree $d-1$ in $Q'$ and $c$ colored neighbors in $Q'$. Thus, at least $d-c$ leaks are in $Q$ since each vertex in $Q$ has only one neighbor in $Q'$. There are $d-2-(d-c) = c-2 \geq 0$ remaining leaks that can be in either $Q$ or $Q'$. %The leak that caused $v_1$ to remain uncolored is distance two from all $x_i$, since that leak is a neighbor of $v_1$ and $v_1$ is a neighbor of all $x_i$. The other $d-1-c$ leaks in $Q$ must be at least distance two from any $x_i$ since each $x_i$ was forced after one application of the color change rule. Since only one vertex is needed to force $v_1$, we need only show that one of the $x_i$ has $v_1$ as its only uncolored neighbor and is not leaky.

Since $v_1$ has $c$ colored neighbors and there can be at most $c-2$ leaks in $Q'$, at least two of these neighbors cannot be leaky. Let $x_1, ..., x_k$ be the $k \geq 2$ colored neighbors of $v_1$ that are not leaky. There are $k$ neighbors of $v_1$ that are not leaky, thus there are $c-k$ that are leaky, so $c-2-(c-k) = k-2$ leaks remain in $Q_d$.  We will now show that $v_1$ is the only uncolored neighbor in $Q'$ of at least one of $x_1, ..., x_k$. 

%\rh{iron out details from here}

%Each of $x_2, ..., x_{k+1}$ has degree $d-1$ in $Q'$ and is a neighbor of $v_1$, so there are $k(d-2)$ vertices in 
Let $A=\cup_{i=1}^{k}N(x_i) \cap Q' \setminus v_1$, where $N(x_i)$ denotes the open neighborhood of $x_i$. 

\begin{claim}\label{claim}
There are $k(d-2) - \binom{k}{2}$ distinct vertices in $A$. 
\end{claim}

\begin{proof}[Proof of Claim~\ref{claim}]
Select a vertex $x_i$. Then $|(N(x_i)\setminus v_1) \cap Q'| =d-2$ since each $x_i$ has degree $d-1$ in $Q'$ and each is a neighbor of $v_1$. Thus, $|A| = k(d-2) - \sum_{i \neq j, i,j \in [k]}| N(x_i) \cap N(x_j)|$. To determine $| N(x_i) \cap N(x_j)|$, first note that for all $i \neq j$, the sequences $x_i$ and $x_j$ differ in two positions, say positions one and two.  The sequence that matches $x_i$ in position one and $x_j$ in the rest of the positions is contained in $N(x_i) \cap N(x_j)$, as is the sequence that matches $x_j$ in position one and $x_i$ in the rest of the positions. These are the only two sequences in $N(x_i) \cap N(x_j)$. Thus, $|N(x_i) \cap N(x_j)| =2$, and $v_1$ must be in each intersection. Now $v_1$ has already been subtracted from the $k(d-2)$ term, so it does not need to be subtracted again. Thus, we only need to subtract one for each pair of colored vertices in the neighborhood of $v_1$. Since there are $k$ colored neighbors of $v_1$,  $\binom{k}{2}$ must be subtracted from $k(d-2)$.
 \end{proof}

Continuing the proof of Theorem~\ref{thm:cube}, if at least $k(d-3)- \binom{k}{2} + 1$ elements of $A$ are colored, this implies that one of the $x_i$ has all colored neighbors except for $v_1$, and thus forces $v_1$. To show that $k(d-3)- \binom{k}{2} + 1$ elements of $A$ are colored, first note that at most $k-2$ vertices of $A$ might not be colored since there were $k-2$ leaks whose locations are unknown. Thus, showing $k(d-2) - \binom{k}{2} - (k-2) > k(d-3) - \binom{k}{2} + 1$ completes the proof. To see this inequality, note that $kd-3k+ 2 > kd - 3k + 1$ since $k$ and $d$ are positive integers. Then, $kd - 2k - k + 2 = k(d-2) - (k-2) > k(d-3) + 1$, so $|A|- (k-2) = k(d-2) + \binom{k}{2} - (k-2) > k(d-3) +\binom{k}{2} +  1$, as required. \qed

Since Dillman and Kenter \cite{dillman2019leaky} proved that $Z_{(\ell)}(G) \leq Z_{(j)}(G)$ for $\ell \leq j$, it immediately follows that $Z_{(\ell)}(Q_d) \leq 2^{d-1}$ for every $\ell \leq d-2$.

\section{Proof of \ref{thm:prism}}\label{sec:prism}
%The prism graph $GP(n,1)$ is a special case of a generalized Petersen graph. This graph consists of two cycles that have $n$ vertices each. Label one vertex of the first cycle $v_1$ and another vertex of the second cycle $u_1$, and increase the index of each vertex by one clockwise around the cycle. The remaining edges in $GP(n,1)$ are $u_iv_i$. An example is shown in Figure~\ref{fig:prism}. We now prove Theorem~\ref{thm:prism}.

%\begin{proof}
Recall that in $GP(n,1)$, we label the vertices of one of the disjoint $n$-cycles $U = u_1, ... , u_n$, we label the vertices of the other $n$-cycle as $X =x_1, ... , x_n$, and all $n$ edges $u_ix_i$ are placed in the graph. 

We shall first prove $Z_{(\ell)}(GP(3,1)) = 3$ for $ l \in \{0,1\}$. We will show $X$ is a minimum $\ell$-leaky forcing set. Color $X$ blue. Since there is at most one leak, two or three vertices in $U$ will be forced after one application of the color change rule. If three are forced, the game is over. If two are forced, they each have exactly one uncolored neighbor and cannot have a leak, so the last vertex is forced after one more application of the color change rule. Thus, $X$ is a $\ell$-leaky forcing set. To show that $Z_{(\ell)}(GP(3,1))$ cannot be less than three, color any two vertices. Since each vertex has degree three, they both have at least two uncolored neighbors, and thus cannot force any vertex.

Now, we show  $Z_{(\ell)}(GP(n,1)) = 4$ for $ \ell \in \{0,1\}$ and $n \geq 4$. Color $u_i, u_{i+1}, x_i$ and $x_{i+1}$  where $i \in \{0,1,2, ..., n\}$ and arithmetic is performed modulo $n$. These vertices each have one uncolored neighbor, so they can force unless they have a leak. However, since at most one vertex has a leak, one vertex of $U$ does not have a leak and has a neighbor  in $X$ cycle that does not have a leak. These two vertices can force around the cycle until they arrive at the leak, in which case everything will have been forced. To show three is not sufficient, color any three vertices. They all have degree three, so at most one of these vertices can force. If it has a leak, however, it cannot, and the graph remains uncolored.

%To prove $Z_{(2)}(GP(3,1)) = 4$, color $u_i, v_i, u_{i+1}$ and $v_{i+1}$ where $i \in \{0,1,2\}$ and arithmetic is performed modulo three. There are three possible ways to place leaks on colored vertices: either both are on $v_i$ and $v_{i+1}$ (or by symmetry $u_i$ and $u_{i+1}$), one is on $u_i$ and the other on $v_{i+1}$ (by symmetry, $v_i$ and $u_{i+1}$), or one is on $u_i$ and the other on $v_{i}$ (by symmetry, $u_{i+1}$ and $v_{i+1}$). In the first case, $u_{i+1}$ forces $u_{i+2}$, which then forces $v_{i+2}$ since $u_{i+2}$ cannot have any leaks. In the second case, $v_{i}$ forces $v_{i+2}$ and $u_{i+1}$ forces $u_{i+2}$, since neither of these have leaks. In the third case, if the leaks are on $u_i$ and $v_i$, $u_{i+1}$ and $v_{i+1}$ force the remaining two uncolored vertices. There is one way to place leaks only on uncolored vertices, and one way to place them on a colored vertex and an uncolored vertex, up to symmetry. If both leaks were on uncolored vertices, they were forced after the first application of the color change rule. If one was placed on a colored vertex, and one on an uncolored vertex, both vertices were forced after the first application of the color change rule since they each have two colored neighbors whose neighborhoods contain one uncolored vertex, ad at most one of them had a leak. 

Since the degree of all vertices of $GP(n,1) = 3$, when $\ell \geq 3$, $Z_{(\ell)}(GP(n,1)) = 2n$ follows from the following result by Dillman and Kenter \cite{dillman2019leaky}. 
\begin{proposition}\label{prop:leakyset}\cite{dillman2019leaky}
Every $\ell$-leaky forcing set in a graph contains all vertices of degree at most $\ell$.
\end{proposition}

%This is due to the fact that each vertex in $GP(n,1)$ has degree three, and $\big \lceil \frac{n}{2} \big \rceil  + 1 \geq 3$ for $n \geq 3$.

To prove $Z_{(2)}(GP(3,1)) = 4$, we first show $Z_{(2)}(GP(3,1)) > 3$ and then give a construction where four vertices in $B$ can force $GP(3,1)$. First, suppose $|B|=3$ for $GP(3,1)$. There are two cases: either all three blue vertices lie on $U$ (by symmetry $X$), or two lie on $U$ and the third is a member of $X$ (by symmetry, two lie on $X$ and one on cycle $U$). If all three lie on $U$ and two have leaks, only one vertex of $X$ is colored after one application of the color change rule. This newly colored vertex has two uncolored neighbors, so it cannot force them. In the second case, if the two vertices on $U$ both have leaks, no vertices can be forced after applying the color change rule, since the only vertex in $B$ without a leak is in $X$ and it has two uncolored neighbors. To show equality, we give a construction where $|B|=4$ and show it can always force the entire graph. Color $u_i, x_i, u_{i+1}$ and $x_{i+1}$ where $i \in \{0,1,2\}$ and arithmetic is performed modulo three. There are three possible ways to place leaks on colored vertices: either both $x_i$ and $x_{i+1}$ are leaky (by symmetry $u_i$ and $u_{i+1}$), $u_i$ and $x_{i+1}$ are leaky (by symmetry, $x_i$ and $u_{i+1}$), or $u_i$ and $x_{i}$ are leaky (by symmetry, $u_{i+1}$ and $x_{i+1}$). In the first case, $u_{i+1}$ forces $u_{i+2}$, which then forces $x_{i+2}$ since $u_{i+2}$ cannot have any leaks. In the second case, $x_{i}$ forces $x_{i+2}$ and $u_{i+1}$ forces $u_{i+2}$, since neither of these have leaks. In the third case, if the leaks are on $u_i$ and $x_i$, $u_{i+1}$ and $x_{i+1}$ force the remaining two uncolored vertices. %There is one way to place leaks only on uncolored vertices, and one way to place them on a colored vertex and an uncolored vertex, up to symmetry. If both leaks were on uncolored vertices, they were forced after the first application of the color change rule. If one was placed on a colored vertex, and one on an uncolored vertex, both vertices were forced after the first application of the color change rule since they each have two colored neighbors whose neighborhoods contain one uncolored vertex, and at most one of them had a leak. 

To show that $Z_{(2)}(GP(4,1)) \leq 6$, color $X$, $u_1$, and $u_2$ blue. Note that $u_3$ and $u_4$ have distinct neighborhoods and each have two neighbors that are colored. If the leaks appear on both colored neighbors of $u_3$, then $u_4$ is forced, which can then force $u_3$. If the leaks appear on one neighbor of both $u_3$ and $u_4$, their other colored neighbor can force them after one step. If the leaks appear on $u_3$ and $u_4$, they are forced immediately, and if a leak is placed on $u_3$ and one colored vertex, both $u_3$ and $u_4$ are forced immediately, as well.

Finally, we will show that $Z_{(2)}(GP(n,1)) \leq n$ for $n > 4$. To do this, color $X$ blue. If both leaks are on $u_i$ and $u_j$, all vertices are forced after one step. If one leak is on $x_i$ and one on $u_j$ for some $i$ and $j$, then all $u_k$ are forced after one step except $u_i$. This vertex will be forced at the next step since it has three colored neighbors, at most two of which can have a leak. If the leaks are on $x_i$ and $x_j$ for some $i \neq j$, there exists $x_k$ and $x_{k+1}$ without leaks. These two can force $u_k$ and $u_{k+1}$, which can then begin the forcing process around the cycle of $u$ vertices. \qed

%\end{proof}

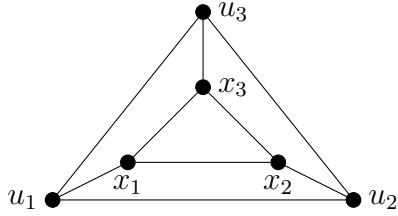
\begin{figure}
\begin{tikzpicture}
	\begin{scope}[every node/.style={scale=0.5,circle,fill=black,draw}]
    \node (A) at (0,0) {};
    \node (B) at (2,0) {};
	\node (C) at (1,1) {}; 
	\node (D) at (-1,-0.5) {};
	\node (E) at (1,2) {}; 
	\node (F) at (3,-0.5) {};

\end{scope}
		
		\begin{scope}[every node/.style={scale=1}]
\node at (-1.4,-.5) {$u_1$};
\node at (3.4,-0.5) {$u_2$};
\node at (1.4,2) {$u_3$};
\node at (0,-0.3) {$x_1$};
\node at (2,-.3) {$x_2$};
\node at (1.4,1) {$x_3$};

\end{scope}
%\draw    (A) to[out=-20,in=20] (C);
%	\draw    (B) to[out=-20,in=20] (D);

	\draw (A) -- (B);
	\draw (A) -- (C);
	\draw (B) -- (C);
	\draw (A) -- (D);
	\draw (B) -- (F);
	\draw (C) -- (E);
	\draw (D) -- (E);
	\draw (D) -- (F);
	\draw (F) -- (E);

\end{tikzpicture}
\caption{The prism graph $GP(3,1)$.}\label{fig:prism}
\end{figure}

\section{Results related to Question \ref{mainquestion}}\label{sec:question}

By Proposition~\ref{prop:leakyset}, $B \subset B_{\ell}$ for graphs where each vertex has degree at most $\ell$, since every vertex must be in $B_{\ell}$. In this section, we shall study some families of graphs that have vertices with degree more than $\ell$ and prove that they satisfy Question \ref{mainquestion}. 

\subsection{Trees}\label{subsec:trees}
We begin the study of Question \ref{mainquestion} with trees. Proposition ~\ref{prop:leakyset} has the following immediate consequence.

\begin{corollary}\label{cor:trees}
For $\ell \geq 1$, a minimum zero-forcing set of a tree is contained in a minimum $\ell$-leaky forcing set.
\end{corollary}

\begin{proof}
The AIM Minimum Rank– Special Graphs Work Group \cite{work2008zero} showed that a minimum zero-forcing set of a tree is a subset of the leaves. Since leaves have degree $1 \leq \ell$, they must be contained in a minimum $\ell$-leaky forcing set by Proposition ~\ref{prop:leakyset}.
\end{proof}

\subsection{Complete bipartite graphs}\label{subsec:completebipartite}

Recall the complete bipartite graph $K_{m,n}$ consists of two sets of vertices, $X,Y$ such that $|X| = m, |Y| = n$. Then $xy \in E(K_{m,n})$ if and only if $x \in X, y \in Y$. We shall use this notation throughout this subsection and assume without loss of generality that $m \geq n$. We show the leaky forcing sets for $K_{m,n}$ have the following property. 

\begin{corollary}\label{cor:completebipartite}
For $\ell \geq 1$, a minimum zero-forcing set of $K_{m,n}$ is contained in a minimum $\ell$-leaky forcing set.
\end{corollary}

Before proving this result, we will first prove the following result.

\begin{proposition}\label{prop:completebipartite}

\begin{align*}
    Z_{(\ell)}(K_{m,n}) = & \begin{cases}
  \; m+n-2 \text{ if } \textup{ $\ell \leq n-1$} \\ 
  \;m+n-1 \text{ if } \textup{ $ m - 1 \geq \ell > n-1$ } \\
  \;m+n \text{ if } \textup{ $ \ell > m - 1$}
\end{cases}\\
\end{align*}
\end{proposition}

\begin{proof}

 First note that if more than one vertex in $X$ is uncolored, no vertex in $X$ be forced, as every vertex in $Y$ is adjacent to all $X$ and only a vertex in $Y$ can force a vertex in $X$. By symmetry, at most one vertex of $Y$ can be uncolored in the initial leaky-forcing set. Thus, $ Z_{(\ell)}(K_{m,n}) \geq m+n-2$ for all $\ell$.

For the first case, color all but one vertex in each set. Since $\ell\leq \min\{m,n\}-1$, at least one vertex in both $X$ and $Y$ do not have leaks. After the first application of the color change rule, there are two outcomes: either all vertices are colored, or all but one are colored. All vertices are colored if there is at least one leak in each set, since this will leave a colored vertex without a leak in each set. If one vertex is not colored, all of the leaks were on colored vertices in the same set, without loss of generality $X$. Additionally, the uncolored vertex of $X$ does not have a leak since there are at most $n-1$ leaks. The uncolored vertex of $X$ is colored after one application of the color change rule and it then forces the uncolored vertex of $Y$ after the second application of the color change rule. 

Now let us consider the second case. First note that all vertices in $Y$ must be colored by Proposition~\ref{prop:leakyset}, since they will have degree $n \leq \ell$, so this implies $ Z_{(\ell)}(K_{m,n}) \geq m+n-1 $.  Next note that the second case only occurs if $m > n$. Since $m>n$, there is at least one vertex of $Y$ that does not have a leak, so it can force the uncolored vertex of $X$. 

The last case trivially holds from Proposition~\ref{prop:leakyset} since each vertex has degree at most $m$.

\end{proof}

%We now prove Corollary \ref{cor:completebipartite}
\begin{proof}[Proof of Corollary \ref{cor:completebipartite}]
 First note that the minimum zero-forcing set of $K_{m,n}$ consists of all but one vertex in $X$ and all but one vertex in $Y$. Let $B$ be a minimum zero-forcing set of $K_{m,n}$. $B$ is a minimum $\ell$-leaky forcing set for $\ell \leq n-1$, $Y \cup (B \cap X)$ is a $\ell$-leaky forcing set for $ m - 1 \geq \ell > n-1$, and the set of all vertices is the minimum $\ell$-leaky forcing for $ \ell > m - 1$. All of these sets contain $B$.
\end{proof}

\subsection{Wheel graphs}\label{subsec:wheel}

Our final aim in this section is to consider the leaky forcing number of wheels. For $n \geq 3$, the wheel $W_n$ is obtained from an $n$-cycle by adding one vertex to the graph (called the central vertex) and adding $n$ edges that join the central vertex to each vertex of the cycle. A minimum zero-forcing set for this graph consists of two adjacent cycle vertices and the central vertex. The following result is due to Dillman and Kenter \cite{dillman2019leaky}.

\begin{proposition}\label{prop:wheel}

\begin{align*}
    Z_{(\ell)}(W_n) = & \begin{cases}
  \; 3\text{ if } \textup{ $\ell\in \{0,1\}$} \\ 
  \;\lceil \frac{2n}{3} \rceil  \text{ if } \textup{ $\ell=2$ } \\
  \;n \text{ if } \textup{ $ n > \ell>2$}\\
  \;n+1 \text{ if } \textup{ $ \ell \in \{n, n+1\}$}.
\end{cases}\\
\end{align*}
\end{proposition}

Their proof of the above result is constructive and they show a minimum $\ell$-leaky forcing set of $W_n$ consists of two adjacent cycle vertices and the central vertex in all cases except when $ n > \ell>2$. In their proof for that case, they color all cycle vertices, instead. Three adjacent cycle vertices, however, is a minimum zero-forcing set for $W_n$, as well, since the central vertex is colored after the first application of the color change rule, and the rest of the cycle vertices can then be forced in successive iterations of the color change rule. Thus,

\begin{corollary}\label{cor:wheels}
A minimum zero-forcing set of $W_n$ is contained in a minimum $\ell$-leaky forcing set.
\end{corollary}

\section{Conclusion}\label{sec:conclusion}

In this paper, we first provided an upper bound for the $(d-2)$-leaky forcing number of $Q_d$, however, the lower bound for $Z_{(d-2)}(Q_d)$ is not known. The partial results on $Q_n$ from \cite{dillman2019leaky} combined with Theorem~\ref{thm:cube}, leads us to believe that 

\begin{conjecture}
$Z_{(d-2)}(Q_d) = 2^{d-1}$. 
\end{conjecture}

Dillman and Kenter showed this conjecture holds for $d \in \{2,3,4,5\}$ in \cite{dillman2019leaky}, so it remains to show it holds when $d \geq 6$. %A remaining open problem is 
 
% \begin{problem}
% Determine $Z_{(l)}(Q_d)$ when $l \neq d-2$.
 %\end{problem}
 
 %Dillman and Kenter mention that for small $d$, some of these values do not appear to follow a nice pattern. \cite{dillman2019leaky}. A computer search on the next few values of $d$ would help provide intuition into a potential pattern.
 
We then found $Z_{(\ell)}(GP(n,1))$ for select $n$ and $\ell$ and provided bounds for $Z_{(2)}(GP(n,1))$ when $n \geq 4$. Determining $Z_{(2)}(GP(4,1))$ and $Z_{(2)}(GP(n,1))$ for $n > 4$ is still open, as is finding the $\ell$-leaky forcing number of $GP(n,k)$ for $k \geq 2$ 
 
  \begin{problem}
 Determine $Z_{(\ell)}(GP(n,p))$ when $p \geq 2$,
 \end{problem}
 %as does 
 
  %\begin{conjecture}\label{conj:prism}
 %$Z_{(2)}(GP(4,1)) = 5$, $Z_{(2)}(GP(n,1)) = n$ for $4< n \leq 7$, and $Z_{(2)}(GP(n,1)) = n-1$ for $n \geq 8$.
 %\end{conjecture}
 
% We believe the first equality in Conjecture~\ref{conj:prism} holds since $B$ with $|B| = 4$ does not appear enough to force $GP(4,1)$ with two leaks, however $B$ with $|B| = 5$ does when vertices $u_1, u_3, u_4, v_1,$ and $v_2$ are colored. Rigorous proof of this is still needed, however.
 
 Finally, we showed that trees, complete bipartite graphs, and wheels satisfy Question~\ref{mainquestion}. Future work includes proving the statement in Question \ref{mainquestion} is true, or finding a counterexample. %It is noted in \cite{alameda2020leaky} that there do exist minimum $l$-leaky-forcing sets that do not contain minimum zero-forcing sets, however in the example they provided, there is an alternative $l$-leaky-forcing set that contains a minimum zero-forcing set. %Characterizing graphs that satisfy Question~\ref{mainquestion} would aid in this process. 
 Clearly if $\ell = |V(G)|$, all vertices must be in $B_{\ell}$, and therefore $B \subset B_{\ell}$. Thus, $\ell$ must be strictly less than $|V(G)|$ in any non-trivial counterexample.  If a zero-forcing set of $G$ consists solely of vertices of degree $\ell$ or less, then the zero-forcing set is guaranteed to be a subset of the $\ell$-leaky-forcing set. Thus, a result that would partially answer Question~\ref{mainquestion} is 
 
 \begin{problem}
Characterize all graphs $G$ such that all elements of the zero-forcing set $B$ have degree at most $\ell$.
 \end{problem}
However, this does not completely resolve Question~\ref{mainquestion}, as  some minimum leaky forcing set may contain vertices with degree greater than $\ell$, such as $K_{m,n}$ with $\ell < \min\{m,n\} -1$. 

Proving or providing a counterexample to the following questions related to Question~\ref{mainquestion} would be of interest, as well.

\begin{question}\label{newquestion}
For a given graph, does there exist a minimum $\ell$-leaky forcing set $B_{\ell}$ such that $B_{\ell}$ contains a minimum $(\ell-k)$-leaky forcing set $B_{\ell-k}$ for some (or all) $\ell \geq k \geq 1$?
\end{question}

\begin{question}\label{newquestion}
For a given graph $G$, does there exist a sequence of nested subsets of vertices $B_0 \subset B_1 \subset ... \subset B_k$ such that $B_{\ell}$ is a minimum $\ell$-leaky forcing set of $G$?
\end{question}

\section*{Acknowledgements}
The author would like to thank B\'ela Bollob\'as for providing helpful comments related to this manuscript.
\bibliographystyle{abbrv}
\bibliography{Manuscript}

\end{document}